\documentclass[a4paper,reqno]{amsart}

\usepackage[a4paper,width={16cm},left=2.5cm,bottom=3cm, top=3cm]{geometry}
\usepackage{enumerate}
\usepackage{yfonts}
\usepackage{xfrac}
\usepackage[english]{babel}

\usepackage{a4wide}
\usepackage[dvipsnames]{xcolor}
\usepackage{amsmath}
\usepackage{amssymb}
\usepackage{amsfonts}
\usepackage{amsthm,graphicx}
\usepackage{comment}
\usepackage{csquotes}
\usepackage{esint}
\usepackage{hyperref}
\hypersetup{
	linktocpage=true,
	colorlinks=true,
	linkcolor=blue!70!black,
	citecolor=orange!40!red,
	urlcolor=magenta!80!black,
}

\usepackage{mathtools}
\mathtoolsset{showonlyrefs}

\numberwithin{equation}{section}
\newtheorem{theorem}{Theorem}[section]
\newtheorem{lemma}[theorem]{Lemma}

\newtheorem{proposition}[theorem]{Proposition}
\theoremstyle{definition}  
\newtheorem{definition}[theorem]{Definition}

\newtheorem{remark}[theorem]{Remark}

\newcommand{\mc}{\mathcal}
\newcommand{\mb}{\mathbb}

\newcommand{\la}{\lambda}

\newcommand{\pd}[2]{\frac{\partial#1}{\partial#2}}

\newcommand{\R}{\mb{R}}

\newcommand{\Tr}{\mathop{\rm{Tr}}}
\newcommand{\dive}{\mathop{\rm{div}}}

\newcommand{\La}{\Lambda}

\usepackage{subfig}
\usepackage{tikz}
\usepackage{xcolor}
\usepackage{pgfplots}
\pgfplotsset{compat=1.18}
\usepackage{mathrsfs}
\usetikzlibrary{arrows}
\usetikzlibrary{shadings}
\makeatletter

\begin{document}
\title[A rigidity result for a global vectorial free boundary problem]
{A rigidity result for a global vectorial free boundary problem with a linear datum}

\author[G. Siclari and B. Velichkov]{Giovanni Siclari and Bozhidar Velichkov}

\address{Bozhidar Velichkov
\newline \indent Dipartimento di Matematica
\newline \indent Universita di Pisa 
\newline\indent Largo Bruno Pontecorvo, 5, 56127 Pisa, Italy}
\email{bozhidar.velichkov@unipi.it}

\address{Giovanni Siclari 
\newline \indent Centro di Ricerca Matematica Ennio De Giorgi
\newline \indent Scuola Normale Superiore di Pisa
\newline\indent Piazza dei Cavalieri 3, 56126 Pisa, Italy}
\email{giovanni.siclari@sns.it}

\date{\today}

\begin{abstract}
In this paper we completely classify the minimizers of a measure constraint global vectorial free boundary problem with a linear datum. More precisely, we show that there exists a unique minimizer whose coincidence set is an ellipsoid. We are motivated by the connection of the problem with the blow-ups of the vectorial Bernoulli free boundary problem and the obstacle problem.
\end{abstract}

\maketitle


\medskip 


\section{Introduction}\label{sec_intro}

Given $d \ge 2$ and $k \ge d$, we consider the functional space
\begin{equation}\label{def_D12}
D^{1,2}(\R^d, \R^k):=
\begin{cases}
\{V \in L^{2^*}(\R^d,\R^k): \nabla V\in L^2(\R^d,\R^{k,d})\}, & \text{ if } d \ge 3,\\
\{V \in L^{2,\infty}(\R^2,\R^k): \nabla V\in L^2(\R^2,\R^{k,2})\}, & \text{ if } d = 2,
\end{cases}
\end{equation}
where $L^{2,\infty}(\R^2,\R^k)$ is the Lorentz space  
\begin{equation}\label{def_L2infty}
L^{2,\infty}(\R^2,\R^k):=\{V:\R^2 \to \R^k \text{ measurable and  such that } [V]_{L^{2,\infty}(\R^2,\R^k)}<+\infty\},
\end{equation}
with 
\begin{equation}\label{def_quasinorm_L2infty}
[V]_{L^{2,\infty}(\R^2,\R^k)}:=\sup_{\la >0} \la |\{|V|>\la\}|^{1/2}.
\end{equation}
For any $d\ge 2$ and any $k\ge 1$, we have the inclusion $D^{1,2}(\R^d,\R^k)\subset  H^1_{loc}(\R^d,\R^k)$.\medskip

Given a $k \times d$ matrix $A \in \R^{k,d}$ with 
$${\rm{rk}}(A)=d,$$  
we are interested in the variational problem 
\begin{equation}\label{prob_min}
 \La^*(A):=\inf\left\{\int_{\R^d} |\nabla V|^2\, dx:V\in D^{1,2}(\R^d, \R^k),  |\{V(x)= Ax\}|=1\right\},
\end{equation}
which we introduced \eqref{prob_min} in \cite{SV_blowups} in order to classify the global singular homogeneous solutions of the vectorial Bernoulli free boundary problem (we discuss the connection with the vectorial Bernoulli problem in Subsection \ref{s:intro:sub:motivations}). In this paper, we compute the exact value of $\Lambda^\ast(A)$ as an analytic expression in terms of the coefficients of $A$. Our main result is the following: 
\begin{theorem}\label{theo_main_1}
Let $d \ge 2$, $k \ge d$,  $A \in \R^{k,d}$ be a $k \times d$ matrix with ${\rm{rk}}(A)=d$.  Then
\begin{equation}\label{eq_La_theo_main}
\La^*(A)=(\Tr(\sqrt{A^TA}))^2.
\end{equation}
\end{theorem}
In order to prove this  theorem, we  characterize the minimizers of \eqref{prob_min}. We prove that,  in any dimension $d\ge2$ and for any $k\times d$ matrix $A$ of rank $d$, there exists a unique minimizer $V$ to \eqref{prob_min} and that its contact set $\{V(x)=Ax\}$ is an ellipsoid.  
To give the complete statement, we need to fix some notation and do some preliminary considerations.\medskip 

\noindent\underline{\it About $A^TA$.} First of all, we notice that if $\text{rk}(A)=d$, then
\begin{equation}\label{e:ATA-properties}
\text{$A^TA \in \R^{d,d}$ is symmetric, positive definite and invertible.} 
\end{equation}
Indeed, the positivity of $A^TA$ follows from the observation that $A^TAx \cdot x=|Ax|^2$ for any $x\in\R^d$. The same formula implies that $\text{Ker}(A)=\text{Ker}(A^TA)$; in particular, if $\text{Ker}(A^TA)\neq\{0\}$, then $\text{rk}(A)<d$, which proves that for matrices satisfying $\text{rk}(A)=d$, we have that $A^TA$ is invertible.\medskip

\noindent\underline{\it Definition of $S$ and $U$.} Thanks to \eqref{e:ATA-properties}, we can define the $d\times d$ and $k\times d$ matrices 
\begin{equation}\label{def_SU}
S:=(A^TA)^{1/2} \in \R^{d,d} \quad \text{ and } \quad  U:=AS^{-1} \in \R^{k,d}.
\end{equation}
We notice that, by construction 
\begin{center}
$S\in \mathbb R^{d\times d}$ is symmetric, positive and invertible. 
\end{center}
Furthermore,
$U^TU=S^{-T}A^TAS^{-1}=S^{-1}S^2S^{-1}={\rm{\mathop {Id}}_d}$ so that $U:\R^d\to\R^k$ is an isometry.\medskip

\noindent\underline{\it Definition of $B$ and $R$.} Let us define the matrix
\begin{equation}\label{def_BRD}
B:=\frac{S}{\Tr(S)}\quad  \text{ and let us write } B=RDR^T,
\end{equation}
where $D \in \R^{d,d}$ is a diagonal matrix and $R \in \R^{d,d}$ is an orthogonal matrix.\medskip

\noindent\underline{\it Definition of $E_A$ and $N_{E_A}$.}
Let $\Gamma$ be the fundamental solution of the Laplacian in $\R^d$: 
\begin{equation}\label{def_Gamma}
\Gamma(x)=
\begin{cases}
\displaystyle\frac{1}{d(d-2) \omega_d}|x|^{2-d} &\text{ if } d\ge 3,\\
\displaystyle\frac{1}{2 \pi}\log(1/|x|) &\text{ if } d=2.
\end{cases}
\end{equation}
For any bounded Lebesgue measurable set $E\subset\R^d$, we set 
\begin{equation}\label{def_NE}
N_E(x)=(\chi_{E}\ast N)(x) =\int_E \Gamma(x-y) \, dy,
\end{equation}
and we notice that the function $N_E$ is a weak solution to the equation
\begin{equation}\label{eq_NE}
-\Delta N_E=\chi_{E}\quad\text{in}\quad \R^d.
\end{equation}
For any $a=(a_1,\dots,a_d) \in \R^d$ with $a_i>0$, let us define the ellipsoid centered in $0$
\begin{equation}\label{def_elips}
E(a):=\left\{x \in \R^d:\sum_{i=1}^d \frac{x_i^2}{a_i^2}<1\right\}.
\end{equation}
In Lemmas \ref{lemma_elipsoid_dge3} and \ref{lemma_elipsoid_d2}, we will prove that there exists a  unique ellipsoid $E_A$ centered in $0$ 
with volume $1$ such that, letting $B$ as in \eqref{def_BRD} and  $N_{E_A}$ as in \eqref{def_NE} for $E=E_A$,
\begin{equation}
\nabla N_{E_A}(x)=-Bx \quad \text{ in } E_A.
\end{equation}
Actually, letting $R$ as in \eqref{def_BRD},
\begin{equation}\label{def_EA}
E_A=R E(\omega_d^{-1/d}a^*),
\end{equation}
where the vector $a^* \in (0,+\infty)^d$  is characterized in Lemmas   \ref{lemma_elipsoid_dge3} and \ref{lemma_elipsoid_d2}. If $d=2$ we actually have a completely explicit expression for $a^*$, see \eqref{def_EA_d2}.\medskip

We are now ready to give a precise statement of our second main result.

\begin{theorem}\label{theo_main_2}
Let $d \ge 2$, $k \ge d$,  $A \in \R^{k,d}$ be a $k \times d$ matrix with ${\rm{rk}}(A)=d$. Let $U$ be as in \eqref{def_SU}, $N_{E_A}$ as in  \eqref{def_NE} with $E_A$  as in \eqref{def_EA}. Then, the function 
\begin{equation}\label{eq_Va_theo_main}
V_A:=-(\Tr(\sqrt{A^TA}))U \nabla N_{E_A},
\end{equation}
is the unique minimizer of \eqref{prob_min} and its contact set is precisely the ellipsoid $E_A$, that is:
\begin{equation}\label{eq_concidence_set_theo_main}
\{V_A(x)=Ax\}=\overline{E_A}.
\end{equation}
\end{theorem}

In the following sections, we highlight several motivations to investigate \eqref{prob_min}, in particular the connection with the blow-ups of the vectorial Bernoulli free boundary problem and the obstacle problem.

\subsection{Applications to the vectorial Bernoulli free boundary problem}\label{s:intro:sub:motivations}
The problem \eqref{prob_min} was introduced in \cite{SV_blowups} as a characterization of the singular blow-ups of the vectorial Bernoulli free boundary problem. From this perspective, Theorems \ref{theo_main_1} and \ref{theo_main_2} have some immediate applications in the theory of the vectorial Bernoulli problem. In particular, Theorem \ref{theo_main_1} provides a criteria for the global minimality of a linear function (Theorem \ref{main-cor}) and a first order free boundary condition at the singular points of a vectorial minimizer (Theorem \ref{main-cor-blow-up}), which was previously unknown.\medskip

Let $D$ be an open set in $\R^d$ and let $\Lambda>0$. For every $W\in H^1(D;\R^k)$ we define the vectorial functional 
\begin{equation}\label{e:vectorial-functional-in-D}
J_\Lambda(W,D):=\int_{D} |\nabla W|^2\, dx+\Lambda |\Omega_W\cap D|,
\end{equation}
where for any $W:D\to\R^k$ we set 
\begin{equation}\label{e:positivity-set-defintion}
\Omega_W:=\{W\neq 0\}.
\end{equation}
\begin{definition}[Vectorial minimizers in $D$]\label{def:minimizers-in-D}
Let $\Lambda\ge 0$. We say that $U\in H^1(D,\R^k)$ is a minimizer of the vectorial functional $J_\Lambda$ in the open set $D\subset\R^d$, if  
\begin{equation}\label{prob_min_vec}
J_\Lambda(U,D)\le J_\Lambda(W,D)\quad\text{for every}\quad W\in H^1(D,\R^k)\quad\text{with}\quad W-U\in H^1_{0}(D,\R^k).
\end{equation}
We will also say that $U$ is a minimizer (or variational solution) of the vectorial Bernoulli problem.
\end{definition}
\begin{definition}[Global vectorial minimizers]\label{def:global-solution}
Let $\Lambda\ge 0$. We say that $U\in H^1_{loc}(\R^d,\R^k)$ is a global minimizer of the vectorial functional $J_\Lambda$, if $U$ is a minimizer of the vectorial Bernoulli problem in $B_R$ (in the sense of Defintion \ref{def:minimizers-in-D}) for all $R>0$.  
\end{definition}

By blow-ups we mean the following.

\begin{definition}[Blow-ups]\label{def:blow_up}
Letting $x_0 \in \partial \Omega_{U}$, we say that $U_0 \in H^1_{loc}(\R^d,\R^k)$ is a blow-up for $U$ at $x_0$ if there exists a sequence $r_n \to 0^+$ such that  
$\frac{1}{r_n}U(x_0+r_nx) \to U_0(x)$ uniformly on compact sets in $\R^d$.
We indicate with $\mc{BU}_U(x_0)$  the sets of all  blow-ups of $U$ at $x_0$.
\end{definition}

The properties of the minimizers of the vectorial Bernoulli problem have been studied by several authors (see the survey \cite{TV_survey} for a detailed overview). 
Letting  $U$ be a minimizer of \eqref{prob_min_vec} and  $\Omega^{(\gamma)}_{U}$ be the set of all points having density $\gamma \in [0,1]$,  we can divide the free boundary into three parts:
\begin{align}\label{def_reg}
&\mathrm{Reg}(\partial \Omega_{U}) := \Omega^{(1/2)}_{U} \cap D,  \\  \label{def_sing2}
&\mathrm{Sing}_2(\partial \Omega_{U}) :=  \Omega^{(1)}_{U} \cap \partial  \Omega_{U} \cap D, \\  \label{def_sing1}
&\mathrm{Sing}_1(\partial \Omega_{U}) := (\partial  \Omega_{U} \cap D) \setminus 
\big( \mathrm{Sing}_2(\partial  \Omega_{U}) \cup \mathrm{Reg}(\partial  \Omega_{U}) \big).
\end{align}
At the moment the main open problem  about the regularity of the free boundary $\partial \Omega_{U}$ concerns $\mathrm{Sing}_2(\partial \Omega_{U})$ about which little is known, see for example \cite{TV_survey}. As detailed in \cite{SV_blowups}, a step towards understanding the structure of the free boundary $\partial\Omega_U$ around points of $\mathrm{Sing}_2(\partial \Omega_{U})$ is to characterize the blow-up limits at such points. A first step in this direction was made in  \cite[Section 2]{MTV_reg_vect}, where it was shown that, if $x_0\in \mathrm{Sing}_2(\partial \Omega_{U})$, then any  $U_0 \in \mc{BU}_U(x_0)$  is a linear map, that is, 
\begin{equation}
U_0(x)= A x \quad \text{ where } A=(a_{i,j}) \in \R^{k,d},
\end{equation}
and a global minimizer  of the vectorial Bernoulli free boundary problem.\medskip

In order to establish which linear maps $Ax$ are global vectorial minimizers, in \cite{SV_blowups} we introduced the quantity 
\begin{equation}\label{def_La}
\Lambda^{\ast\ast}(A)=\sup\{\Lambda>0:Ax \text{ is a global minimizer of } J_\Lambda\text{ in }\R^d\},
\end{equation}
which is well-defined for any non-zero $k\times d$ matrix $A$ independently of its rank.
We notice that this quantity characterizes the values of $\Lambda>0$ for which the linear function $Ax$ is a global minimizer of the vectorial functional $J_\Lambda$ as one can easily check that we have the following dichotomy:
\begin{itemize}
\item if $\Lambda\le \Lambda^{\ast\ast}(A)$, then $Ax$ is a global minimizer of the vectorial functional $J_\Lambda$;
\item conversely, if $\Lambda> \Lambda^{\ast\ast}(A)$, then  $Ax$ is not a global minimizer of $J_\Lambda$.
\end{itemize}
Indeed, when $\Lambda>\Lambda^{\ast\ast}(A)$, $Ax$ is not minimizing by the definition of $\Lambda^{\ast\ast}(A)$. For the case $\Lambda<\Lambda^{\ast\ast}(A)$, it is sufficient to observe that if a linear map $Ax$ minimizes $J_{\La'}$ for some $\La'>0$ then, since $|\Omega_{Ax}\cap B_R|=|B_R|$,  it minimizes $J_{\La''}$ for any $\La'' \le \La'$. Finally, for the limit case $\Lambda=\Lambda^{\ast\ast}(A),$ we only need to observe the set of $\Lambda$ for which $Ax$ is not a global minimizer of $J_\Lambda$ is an open subset of $(0,+\infty)$.\medskip 

In \cite{SV_blowups} we characterized the value of $\Lambda^{\ast\ast}(A)$ defined in \eqref{def_La} in terms of the shape optimization problem \eqref{prob_min}. More precisely, in \cite{SV_blowups}, we proved the following theorem.

\begin{theorem}[\cite{SV_blowups}]\label{theo_La_old_paper}
Let $A$ be a $k \times d$ matrix with ${\rm{rk}}(A)=n$ and  $1\le n\le d$. Let $Q \in \R^{d,d}$ be an orthogonal matrix such that 
$A=\begin{bmatrix} A_1, 0\end{bmatrix}Q$, for some matrix $A_1 \in \R^{k,n}$ of rank $n$.
Then, the following holds: 
\begin{itemize}
\item if ${\rm{rk}}(A)>1$, then 
\begin{equation}
\Lambda^{\ast\ast}(A)=\Lambda^{\ast\ast}(A_1)=\Lambda^{\ast}(A_1), 
\end{equation}
where $\Lambda^\ast(A_1)$ is given by \eqref{prob_min}; in particular, 
\begin{equation}
\Lambda^{\ast\ast}(A)>\Tr(A^TA);
\end{equation}
\item if  ${\rm{rk}}(A)=1$, then
\begin{equation}\label{e:old-teo-rank1}
\Lambda^{\ast\ast}(A)=\Tr(A^TA).
\end{equation}
\end{itemize}
\end{theorem}
We notice that our Theorem \ref{theo_La_old_paper} gave a negative answer to the question 
\begin{equation}\label{e:gap-question}
\text{\it Is it true that 
$\Lambda^{\ast\ast}(A)=\Tr(A^TA)$ even if ${\rm{\mathop {rank}}}(A)>1$?} 
\end{equation}
previously posed in \cite{MTV_reg_vect}, but didn't provide a closed analytic expression for $\Lambda^\ast(A)$.\medskip

Our main theorem from the present paper (Theorem \ref{theo_main_1}) refines  the classification of the blows-up limit at points $x_0 \in \mathrm{Sing}_2(\partial \Omega_{U})$ by providing explicitly the exact value of $\La^{\ast\ast}(A)$. Indeed, we have the following 
\begin{theorem}\label{main-cor}
Let $d \ge 2$, $k \ge d$,  $A \in \R^{k,d}$ be a $k \times d$ matrix with $1\le {\rm{rk}}(A)\le d$.  Then
\begin{equation}\label{eq_La_theo_main_cor}
\La^{\ast\ast}(A)=(\Tr(\sqrt{A^TA}))^2.
\end{equation}
In particular, for any $\Lambda>0$, we have the following dichotomy:
\begin{itemize}
\item if $\sqrt\Lambda\le \Tr(\sqrt{A^TA})$, then $Ax$ is a global minimizer of the vectorial functional $J_\Lambda$ in $\R^d$;
\item if $\sqrt\Lambda> \Tr(\sqrt{A^TA})$, then  $Ax$ is not a global minimizer of $J_\Lambda$.
\end{itemize}
\end{theorem}
\begin{proof}
The proof is a combination of Theorem \ref{theo_main_1} and Theorem \ref{theo_La_old_paper}. We start setting $n:=\text{rk}(A)$. We consider two cases.\medskip

\noindent{Case 1: $n=d$.} By Theorem \ref{theo_La_old_paper}, we have that 
$$\Lambda^{\ast\ast}(A)=\Lambda^\ast(A),$$
while Theorem \ref{theo_main_1} implies that 
$$\Lambda^\ast(A)=(\Tr(\sqrt{A^TA}))^2.$$
By combining these identities, we get \eqref{eq_La_theo_main_cor}.
\medskip

\noindent{Case 2: $n=1$.} Since the rank of $A$ is the same as the rank of the symmetric matrix $A^TA$, we get that $A^TA$ has a unique positive eigenvalue $\lambda>0$. Thanks to \eqref{e:old-teo-rank1}, we have that 
$$\Lambda^{\ast\ast}(A)=\Tr(A^TA)=\lambda.$$
On the other hand, the symmetric $d\times d$ matrix $\sqrt{A^TA}$ has a unique eigenvalue $\sqrt{\lambda}$ and so, we have $(\Tr(\sqrt{A^TA}))^2=(\sqrt{\lambda})^2=\lambda$, which concludes the proof of \eqref{eq_La_theo_main_cor} in this case.\medskip

\noindent{Case 3: $1<n<d$.} There exist an orthogonal $d\times d$ matrix $Q$ and a $k\times n$ matrix $A_1 $ of rank $n$ such that 
$A=\begin{bmatrix} A_1, 0\end{bmatrix}Q$, where with $[A_1,0]$ we denote the $k\times d$ matrix for which the first $n$ columns are the columns of $A_1$ and the last $d-n$ columns are $0$. Thanks to Theorem \ref{theo_La_old_paper}, we have that 
$$\Lambda^{\ast\ast}(A)=\Lambda^{\ast\ast}(A_1)=\Lambda^{\ast}(A_1).$$
Since $A_1$ is a $k\times n$ matrix of rank $n$, we can apply Theorem \ref{theo_main_1}, obtaining 
$$\Lambda^{\ast}(A_1)=\left(\Tr\left(\sqrt{A_1^TA_1}\right)\right)^2.$$
Now, to conclude, we observe that 
\begin{equation*}
A^TA=Q^T[A_1,0]^T[A_1,0]Q=Q^T\begin{pmatrix}
A_1^TA_1&0\\
0&0
\end{pmatrix}Q,
\end{equation*}
and so
\begin{equation*}
(A^TA)^{\sfrac12}=Q^T\begin{pmatrix}
(A_1^TA_1)^{\sfrac12}&0\\
0&0
\end{pmatrix}Q,
\end{equation*}
which in particular implies that
$$\text{Tr}\left((A^TA)^{\sfrac12}\right)=\text{Tr}\begin{pmatrix}
(A_1^TA_1)^{\sfrac12}&0\\
0&0
\end{pmatrix}=\text{Tr}\left((A_1^TA_1)^{\sfrac12}\right),$$
which concludes the proof.
\end{proof}


\begin{remark}
Theorem \ref{main-cor} allows to compute exactly the gap between $\Tr(A^TA)$  and $\Lambda^{\ast\ast}(A)$. Indeed, let $A$ be a $k\times d$ matrix, let $n$ be the rank of $A$ and let $1<n\le d$.  Denoting with $\sigma_i^2$, $i=1,\dots,n$, the eigenvalues of $A^TA$ (counted with their multiplicity), we have that 
\begin{equation}
\Tr(A^TA)=\sum_{i=1}^n \sigma_i^2,
\end{equation}
while by Theorem \ref{main-cor}
\begin{equation}
\Lambda^{\ast\ast}(A)=\left(\sum_{i=1}^n \sigma_i\right)^2.
\end{equation}
Thus 
\begin{equation}
 \Lambda^{\ast\ast}(A)-\Tr(A^TA)=\sum_{i,j=1,i\neq j}^n\sigma_i\sigma_j>0,
\end{equation}
which gives a quantitative answer to the question \eqref{e:gap-question}.
\end{remark}

\begin{theorem}\label{main-cor-blow-up}
Let $D$ be an open set in $\R^d$, let $\Lambda>0$ be a real constant, and let $U\in H^1(D;\R^k)$  be a minimizer of the vectorial function $J_\Lambda$ in $D$. Then, for every $x_0\in Sing_2(\partial\Omega_U)\cap D$ and every blow-up $U_0$ of $U$ at $x_0$ we have the inequality
\begin{equation}\label{e:blow-up-lower-bound}\text{\rm Tr}\left(\sqrt{(DU_0)^T(DU_0)}\right)\ge \sqrt{\Lambda},
\end{equation}
and
\begin{equation}\label{e:blow-up-sequence-independence}
\text{\rm Tr}\left(\sqrt{(DU_0)^T(DU_0)}\right)=\lim_{r\to0^+}\text{\rm Tr}\left(\sqrt{(DU)_{x_0,r}^T(DU)_{x_0,r}}\right),
\end{equation}
where $(DU)_{x_0,r}$ is the $k\times d$ matrix
$$(DU)_{x_0,r}:=\frac{1}{|B_r|}\int_{B_r(x_0)}DU(x)\,dx,$$ 
$DU$ being the $k\times d$ matrix for which the $j$th row is the gradient of the $j$th component $u_j$ of $U$.
\end{theorem}
\begin{proof}
Let $U_{x_0,r_n}(x)=r_n^{-1}U(x_0+r_nx)$ be a sequence converging to a blow-up limit $U_0:\R^d\to\R^k$ uniformly on every compact subset of $\R^d$. We know that (see for instance \cite{MTV_reg_vect}) that $U_{x_0,r_n}$ converges to $U_0$ strongly in $H^1(B_1,\R^k)$ and that $U_0$ is a linear global minimizer of the vectorial problem. In particular, this implies that $(DU)_{x_0,r_n}\to DU_0$ as $n\to+\infty$ and so 
$$\text{\rm Tr}\left(\sqrt{(DU_0)^T(DU_0)}\right)=\lim_{n\to+\infty}\text{\rm Tr}\left(\sqrt{(DU)_{x_0,r_n}^T(DU)_{x_0,r_n}}\right).$$
Now, by Theorem \ref{main-cor} and the definition \eqref{def_La} of $\Lambda^{\ast\ast}$ we have that 
$$\Lambda\le \Lambda^{\ast\ast}(DU_0)=\left(\text{\rm Tr}\left(\sqrt{(DU_0)^T(DU_0)}\right)\right)^2,$$
which proves \eqref{e:blow-up-lower-bound}. 
By \cite[Theorem 1.5]{SV_blowups} we know that $\Lambda^{\ast\ast}$ depends only on the point $x_0$ and not on the blow-up sequence, which proves \eqref{e:blow-up-sequence-independence}.
\end{proof}

\begin{remark}
Theorem \ref{main-cor-blow-up} implies that if $U:D\to\R^k$ is a minimizer of the vectorial functional $J_\Lambda$ from \eqref{e:vectorial-functional-in-D} in the open set $D\subset\R^d$, then $U$ solves the overdetermined system 
\begin{equation}\label{e:vectorial-system}\begin{cases}
\Delta U=0.&\text{in}\quad \Omega_U:=\{U\neq 0\}\cap D, \smallskip\\
|\nabla U|=\sqrt{\Lambda}.&\text{on}\quad \text{\rm Reg}(\partial\Omega_U),\smallskip\\
\text{\rm Tr}\Big(\sqrt{(DU)^T(DU)}\Big)\ge\sqrt{\Lambda}.&\text{on}\quad \text{\rm Sing}_2(\partial\Omega_U),
\end{cases}
\end{equation}
where the third inequality holds in the sense explained in Theorem \ref{main-cor-blow-up} above, precisely: 
$$\lim_{r\to0^+}\text{\rm Tr}\left(\sqrt{(DU)_{x_0,r}^T(DU)_{x_0,r}}\right)\ge \sqrt{\Lambda}\quad\text{for all}\quad x_0\in \text{\rm Sing}_2(\partial\Omega_U).$$
We notice that in the scalar case $k=1$, the minima of $J_\Lambda$ are precisely the minima of the classical Alt-Caffarelli-Friedman two-phase functional and that in this case the above system reads as 
\begin{equation}\label{e:two-phase-system}\begin{cases}
\Delta u=0,&\text{in}\quad \Omega_u:=\{u\neq 0\}\cap D,\smallskip\\
|\nabla u|=\sqrt{\Lambda},&\text{on}\quad \text{\rm Reg}(\partial\Omega_u),\smallskip\\
|\nabla u|\ge \sqrt{\Lambda,}&\text{on}\quad \text{\rm Sing}_2(\partial\Omega_u),
\end{cases}
\end{equation}
which is precisely the classical two-phase Bernoulli free boundary problem. We notice that the free boundary conditions in \eqref{e:two-phase-system} can be deduced from the inner variations of the functional $J_\Lambda$ (see for instance \cite{MTV_reg_vect}, where this argument was used at the points of the highest stratum of $\text{\rm Sing}_2(\partial\Omega_U)$). On the other hand, the free boundary condition on $\text{\rm Sing}_2(\partial\Omega_U)$ in \eqref{e:vectorial-system} is a consequence of the analysis of the global minimizes of the variational problem \eqref{prob_min}. This boundary condition reflects the vectorial nature of the functional and is not the consequence of an inner variation argument, but from deformations which correspond to enlarging the zero set $\{U=0\}$ from a $j$-dimensional surface $\Gamma$ to an ellipsoidal tubolar neighborhood of this surface.
\end{remark}

\subsection{Connection with the global obstacle problem}

It was proved in \cite{DF_bubble,D_att,FS_quad,L_inv,S_dim2}, that any solution to the global  obstacle problem
\begin{equation}\label{prob_obs}
\Delta u=\chi_{\{u>0\}} \quad \text{ in } \R^d,
\end{equation}
with $\{u=0\}$ bounded and with nonempty interior, must have as coincidence set $\{u=0\}$ an ellipsoid. For a short proof in dimension $d \ge3$, which also inspired the Hopf-type argument we used in the proof of Theorem \ref{theo_main_2}, see  \cite{EW_short_proof}.
For a more general classification result, that deals also with the case of unbounded contact sets, we refer to \cite{EFW_glob_obs,S_dim2,SW_glob_dim2}.

To understand the connection between \eqref{prob_min} and \eqref{prob_obs}, we give a sketched argument omitting several details.
%
Letting $B$ be as in \eqref{def_BRD}, let us consider 
\begin{equation}\label{def_psi_la}
\psi_\la(x):=\la -x\cdot Bx
\end{equation}
and the obstacle problem 
\begin{equation}
\inf\left\{\int_{\R^d}|\nabla w|^2 dx: w \in D^{1,2}(\R^d,\R), w \ge \psi_\la\right\}.
\end{equation}
It is easy to see that there exists a unique minimizer $v_\la$. Let $E_\la:=\{v_\la=\psi_\la\}$ be its coincidence set. Minimality yields that  $v_\la$ satisfies
\begin{equation}
-\Delta v_\la=\chi_{E_\la},
\end{equation}
thus $u_\la:=v_\la-\psi_\la$  is a solution of \eqref{prob_obs}.
In view of \eqref{eq_NE},  we have 
\begin{equation}
N_{E_\la}=v_\la \quad \text{ and } \quad \nabla N_{E_\la} (x)=-Bx \text{ on } E_\la.
\end{equation}
Since $u_\la \ge 0$ by minimality, it is easy to see that $E_\la$ must be bounded and that $|E_\la|>0$. Solution to \eqref{prob_obs} are known to be convex (see \cite[Theorem 5.1]{PSU_conv}) so that also $E_\la$ is convex and in particular its interior is not empty.  Then, by \cite{EW_short_proof, S_dim2}, $E_\la$ must be an ellipsoid. A scaling argument over $\la$ also gives, for a suitable value  $\la^*$, that $|E_{\la^*}|=1$.

Hence, we may, roughly speaking, see minimizers of \eqref{prob_min} as the gradient of minimizers of the obstacle problem \eqref{prob_obs} after subtracting $\psi_{\la^*}$. Of course, one still need to multiply by $-(\Tr(\sqrt{A^TA}))U$ to recover a minimizer of  \eqref{prob_min}.
Equivalently, we are selecting a minimizer of \eqref{prob_obs} whose coincidence set is determined by  $B$, with $B$ as in \eqref{def_BRD}.

Furthermore, these sketched ideas, once fully detailed,  are an alternative approach to Lemma \ref{lemma_elipsoid_dge3} and  Lemma \ref{lemma_elipsoid_d2}. 
However, to our understanding, with this alternative approach the ellipsoidal coincidence set is not determined explicitly.
Hence, we rather use a more direct argument which is based on a explicit formula for the solution of \eqref{eq_NE} when $E$ is an ellipsoid, see \cite{CMMRSV_d2,DF_dge3}.

\subsection*{Organization of the paper} 
In Section \ref{subsec_ell} we construct a family of  competitors $N_E$, of the form \eqref{def_NE}, with prescribed gradients on the corresponding ellipsoids $E\subset\R^d$. 
In Section \ref{subsec_La} we prove Theorem \ref{theo_main_1} and we show that there exists a minimizer of \eqref{prob_min} whose  contact set is the ellipsoid $E_A$ defined in \eqref{def_EA}. In Section \ref{sec_rig} we show that the minimizer of \eqref{prob_min} is unique.

\subsection*{AI statement}
The computation of the gradient of $N_E$ on the ellipsoid $E$ in dimension $d\ge 3$ (Lemma \ref{lemma_elipsoid_dge3}) is the result of a chat with ChatGPT 5.6 Sol with minimal modifications. All the other results and proofs in the paper were conceived and written by the authors.

\section{Construction of the competitors $N_{E_A}$}\label{subsec_ell}

In this section we build an explicit family of competitors $N_E$ of the form \eqref{def_NE} associated to ellipsoids $E\subset\R^d$. We compute explicitly the gradient of $N_E$ on the ellipsoid $E$ and we show that 
\begin{equation*}
\nabla N_{E}(x)=-Bx \quad \text{ for all }\quad x\in E.
\end{equation*}
In particular, thanks to the explicit formulas for $N_E$ provided by  \cite{CMMRSV_d2,DF_dge3}, we are able to trace the precise relation between the matrix $B$ and the ellipsoid $E$. This   allows to invert the procedure and to find the ellipsoid $E$ for every prescribed  matrix $B$ of the form \eqref{def_BRD}.
This will be the key property to obtain an explicit minimizer.  We distinguish the cases $d\ge3$ and $d=2$ since the case $d=2$ is far simpler and we begin with an elementary lemma.

\begin{lemma}\label{lemma_homo_gaus}
Let $f:\R^d\to\R$ be a $0$-homogeneous function measurable and integrable in $\mb{S}^{d-1}$. Then 
\begin{equation}\label{eq_homo_gaus}
\fint_{\mb{S}^{d-1}}f\, d \mc{H}^{d-1}=\pi^{-d/2}\int_{\R^d} f(x) e^{-|x|^2} \, dx.
\end{equation}
\end{lemma}

\begin{proof}
The proof is a direct computation in polar coordinates,
\begin{align*}
\int_{\R^d} f(x) e^{-|x|^2} \, dx &=\int_{\mb{S}^{d-1}}f\, d \mc{H}^{d-1} \int_0^{+\infty} r^{d-1}e^{-r^2}  \, dr\\
&=\fint_{\mb{S}^{d-1}}f\, d \mc{H}^{d-1} \int_{\R^d} e^{-|x|^2}  \, dx=\pi^{d/2}\fint_{\mb{S}^{d-1}}f\, d \mc{H}^{d-1}.\qedhere
\end{align*}
\end{proof}

\begin{lemma}\label{lemma_elipsoid_dge3}
Let $d\ge 3$, $k\ge d$, and let $A$ be a $k\times d$ matrix of rank $d$. Let $S:=(A^TA)^{1/2}$ and $B:=(\text{\rm Tr}(S))^{-1}S$.
Then, there exists a  unique ellipsoid $E_A$, centered in $0$ and
with volume $|E_A|=1$, such that
\begin{equation}\label{eq_nabla_NEa_dge3}
\nabla N_{E_A}(x)=-Bx \quad \text{ in } E_A,
\end{equation}
where $N_{E_A}$ is given by \eqref{def_NE}.
More precisely, the ellipsoid $E_A$ is constructed as follows:
\begin{enumerate}[\quad\rm(1)]
\item We first write the matrix $B$ in the form $B=RDR^T$ (as in \eqref{def_BRD}), where $R$ is an orthogonal $d\times d$ matrix and $D$ is a diagonal matrix with entries $\{d_i\}_{i=1,\dots,d}$.\smallskip
\item We then define the vector $a^*=(a^\ast_1,\dots,a^\ast_d) \in (0,+\infty)^d$ as  
\begin{equation}\label{eq_a_star}
a^*_i:=e^{-p_i^*/2} \quad \text{ for }  i=1, \dots, d,
\end{equation}
where $p^*=(p_1^\ast,\dots,p_d^\ast)\in\R^d$ is the (unique) minimum of the map 
\begin{equation}\label{def_PhiD}
\Phi_D:\R^d\to\R\ ,\qquad \Phi_D(p):=\fint_{\mb{S}^{d-1}}\log\left(\sum_{j=1}^de^{p_j}\theta_j^2\right) \, d \mc{H}^{d-1}(\theta)-\sum_{j=1}^d d_j p_j
\end{equation}
on the hyperplane 
\begin{equation}\label{def_H}
H:=\left\{p=(p_1,\dots,p_d) \in \R^d: \sum_{j=1}^dp_j=0\right\}.
\end{equation}
\item Finally, the ellipsoid $E_A$ is defined as 
\begin{equation}\label{def_EA_dge3}
E_A=R E(\omega_d^{-1/d}a^*),
\end{equation}
where $R$ is the orthogonal matrix from (1) and $E(\omega_d^{-1/d}a^*)$ is the ellipsoid associated to the vector $\omega_d^{-1/d}a^*$ as in \eqref{def_elips}.
\end{enumerate}
\end{lemma}

\begin{proof}
First suppose that $B$ is a diagonal matrix, that is $B=D$, with $B,D$ is as in \eqref{def_BRD}. By \cite[Theorem 2.1]{DF_dge3}  for any ellipsoid  $E(a)$ with $a \in (0,+\infty)^d$, denoting with
\begin{equation}
P(a):=\prod_{j=1}^da_j \quad \text{ and } \quad D_a(t):=\prod_{j=1}^d\sqrt{a_j^2+t} \quad \text{ for any } t\in [0,+\infty),
\end{equation}
we have
\begin{equation}\label{proof_lemma_elipsoid_dge3_1}
N_{E(a)}(x)=\frac{P(a)}{4}\int_0^{+\infty}\frac{1}{D_a(t)}\left(1-\sum_{i=1}^d\frac{x_i^2}{a_i^2+t}\right) \, dt
\quad \text{ in } E(a).
\end{equation}
Then, denoting with 
\begin{equation}\label{proof_lemma_elipsoid_dge3_2}
L_i(a):=\frac{P(a)}{2}\int_0^{+\infty}\frac{1}{D_a(t)(a_i^2+t)}\, dt \quad  \text{ for } i=1, \dots, d,
\end{equation}
it follows that
\begin{equation}
\nabla N_{E(a)}(x)=-(L_1(a)x_1, \dots,L_d(a)x_d) \quad \text{ in } E(a).
\end{equation}
We observe some elementary properties of $L_i(a)$. First of all we notice that for any $r>0$ 
\begin{equation}\label{proof_lemma_elipsoid_dge3_3}
L_i(ra)=L_i(a) \quad  \text{ for } i=1, \dots, d.
\end{equation}
Indeed, the change of variables $r^{-2}t=s$ yields
\begin{equation}
L_i(ra)=\frac{r^dP(a)}{2}\int_0^{+\infty}\frac{1}{(r^2a_i^2+t)\prod_{j=1}^d\sqrt{r^2a_j^2+t}}\, dt=L_i(a).
\end{equation}
Furthermore,
\begin{equation}\label{proof_lemma_elipsoid_dge3_4}
\sum_{i=1}^dL_i(a)=1.
\end{equation}
Indeed, 
\begin{equation}
\sum_{i=1}^dL_i(a)=\frac{P(a)}{2}\int_0^{+\infty}\frac{1}{D_a(t)}\sum_{i=1}^d\frac{1}{(a_i^2+t)}\, dt\\
=-P(a)\int_0^{+\infty}\frac{d}{dt}\frac{1}{D_a(t)}\, dt=\frac{P(a)}{D_a(0)}=1.
\end{equation}
The validity of  \eqref{eq_nabla_NEa_dge3} in $E(a)$ is equivalent to $L_i(a)=d_i$ for $i=1, \dots, d$, where $d_i$ are the diagonal entries of the matrix $D$ defined in is  \eqref{def_BRD}.
We claim that  for any $d \in (0,+\infty)^d$ with $\sum_{i=1}^d d_i=1$, there exists a unique $a^* \in (0,+\infty)^d$ such that $L_i(a^*)=d_i$ for $i=1, \dots, d$, $P(a^*)=1$ and that $a^*$ is actually  given by \eqref{eq_a_star}. In view of \eqref{proof_lemma_elipsoid_dge3_4}, such a claim at least makes sense.

To prove this claim, we first set $b_i:=a_i^{-2}$ and rewrite \eqref{proof_lemma_elipsoid_dge3_2} as
\begin{equation}\label{proof_lemma_elipsoid_dge3_5}
L_i(a):=\frac{b_i}{2}\int_0^{+\infty}\frac{1}{(1+b_it)\prod_{j=1}^d\sqrt{1+b_jt}}\, dt \quad  \text{ for } i=1, \dots, d.
\end{equation}
We claim that 
\begin{equation}\label{proof_lemma_elipsoid_dge3_6}
L_i(a)=\fint_{\mb{S}^{d-1}}\frac{b_i \theta_i^2}{\sum_{j=1}^db_j\theta_j^2}\, d \mc{H}^{d-1} \quad  \text{ for } i=1, \dots, d.
\end{equation}
Applying \eqref{eq_homo_gaus} to the function 
$$f(x)=f(x_1,\dots,x_d):=\frac{b_i x_i^2}{\sum_{j=1}^db_jx_j^2},$$   
and using the identity
$$\frac1q=\int_0^{+\infty}e^{-tq}\, dt\quad\text{with}\quad  q=\sum_{j=1}^db_jx_j^2,$$ 
we get
\begin{align*}
\fint_{\mb{S}^{d-1}}\frac{b_i \theta_i^2}{\sum_{j=1}^db_j\theta_j^2}\, d \mc{H}^{d-1}&
=\pi^{-d/2}\int_{\R^d} \frac{b_i x_i^2}{\sum_{j=1}^db_jx_j^2}e^{-|x|^2} \, dx\\
&=\pi^{-d/2}\int_0^{+\infty} b_i\int_{\R^d}x_i^2e^{-\sum_{j=1}^d(1+t b_j)x_j^2}\, dx dt\\
&=\pi^{-d/2}\int_0^{+\infty} b_i\left(\int_{\R} x_i^2e^{-(1+t b_j)x_i^2} \, dx_i\right) \prod_{j=1, j \neq i}^d\left(\int_{\R} e^{-(1+t b_j)x_j^2} \, dx_j\right) dt\\
&=L_i(a),
\end{align*}
since for any $\alpha>0$
\begin{equation}\label{proof_lemma_elipsoid_dge3_7}
\int_{\R} e^{-\alpha x^2} \, dx=\sqrt{\frac{\pi}{\alpha}} \quad  \text{ and } \quad 
\int_{\R} x^2e^{-\alpha x^2} \, dx=\frac{\sqrt\pi}{2\alpha^{3/2}}.
\end{equation}
Note that the second identity in \eqref{proof_lemma_elipsoid_dge3_7} follows from the first taking the derivative with respect to $\alpha$.
In conclusion, we have proved \eqref{proof_lemma_elipsoid_dge3_6}. Setting  $b_i:=e^{p_i}$,
\begin{equation}\label{proof_lemma_elipsoid_dge3_8}
L_i(a)=\fint_{\mb{S}^{d-1}}\frac{e^{p_i} \theta_i^2}{\sum_{j=1}^de^{p_j}\theta_j^2}\, d \mc{H}^{d-1} \quad  \text{ for } i=1, \dots, d
\end{equation}
which implies that, letting $\Phi_0$ be as in \eqref{def_PhiD} with  $d_i=0$,
\begin{equation}
\nabla \Phi_0(p)=(L_1(a), \dots,L_d(a)).
\end{equation}
We claim that $\Phi_0$ is strictly convex in $H$, where $H$ is as in \eqref{def_H}. 
To prove this claim let us set 
\begin{equation}
w_i(p,\theta):=\frac{e^{p_i} \theta_i^2}{\sum_{j=1}^de^{p_j}\theta_j^2}.
\end{equation}
Then, letting $\delta_{i,j}$ the usual Kronecker delta,
\begin{equation}
\pd{w_i}{p_j}=\delta_{i,j}w_i-w_i w_j.
\end{equation}
Hence, since $\sum_{j=1}^d w_j=1$, for any $z \in \R^d$
\begin{align*}
z\cdot D^2\Phi_0 z&=\fint_{\mb{S}^{d-1}}\sum_{i,j=1}^dz_iz_j (\delta_{i,j}w_i-w_i w_j)\, d \mc{H}^{d-1}\\
&=\fint_{\mb{S}^{d-1}}\sum_{i=1}^d w_i z_i^2-\left(\sum_{i=1}^dw_iz_i\right)^2\, d \mc{H}^{d-1}\\
&=\frac{1}{2}\fint_{\mb{S}^{d-1}}2\left(\sum_{i=1}^d w_i z_i^2\right)\left(\sum_{j=1}^d w_j\right)
-2\left(\sum_{i=1}^dw_iz_i\right)^2\, d \mc{H}^{d-1}\\
&=\frac{1}{2}\fint_{\mb{S}^{d-1}}\sum_{i,j=1}^dw_iw_j (z_i-z_j)^2\, d \mc{H}^{d-1} \ge 0,
\end{align*}
and in particular  $z\cdot D^2\Phi_0 z>0$ on $H\setminus \{0\}$.
Next we claim that, letting $\Phi_D$ be as in \eqref{def_PhiD},   $\Phi_D$ is coercive in $H$.
Indeed,
\begin{equation}
\Phi_D(p) \ge \max_{i=1, \dots, d}{p_i}+\fint_{\mb{S}^{d-1}}\log(\theta_1^2)\, d \mc{H}^{d-1}-\sum_{j=1}^d d_j p_j.
\end{equation}
Furthermore, using the definition of $H$ and the fact that $\sum_{i=1}^dd_i=1$, we get that 
\begin{equation}
\sum_{i=1}^dd_ip_i <  \max_{i=1, \dots, d}{p_i}\quad\text{for all}\quad p=(p_1,\dots,p_d)\in H.
\end{equation}
This implies that the continuous function 
$$p=(p_1,\dots,p_d)\mapsto\max_{i=1, \dots, d}{p_i}- \sum_{j=1}^d d_j p_j,$$ 
is bounded away from zero on the compact set $\mb{S}^{d-1} \cap H$. Thus, there is $\delta>0$ such that 
\begin{equation}
 \max_{i=1, \dots, d}{p_i}-\sum_{j=1}^d d_j p_j  \ge \delta |p|\quad\text{for all}\quad p=(p_1,\dots,p_d)\in H.
\end{equation}
Hence, $\Phi_D$ is coercive in $H$.

By strict convexity and coercivity it follows that $\Phi_D$ admits a unique minimum  $p^* \in H$ and 
\begin{equation}
\nabla \Phi_D(p^* )=\la (1, \dots, 1),
\end{equation}
for some Lagrange multiplier $\la \in \R$. Equivalently
\begin{equation}
\nabla \Phi_0(p^* )-(d_1,\dots,d_d)=\la (1, \dots, 1)
\end{equation}
Since $\nabla \Phi_0(p)=(L_1(a),\dots, L_d(a))$ for any $p \in  \R^d$, by \eqref{proof_lemma_elipsoid_dge3_4} and $\sum_{j=1}^d d_j=1$,
we must have $\la=0$, that is,
\begin{equation}
\nabla \Phi_0(p^* )=(d_1,\dots,d_d).
\end{equation}
If we now set $a^*_i:=e^{-p_i^*/2}$ we have, since $p^* \in H$,
\begin{equation}
(L_1(a^*),\dots, L_d(a^*))=(d_1,\dots,d_d) \quad \text{ and } \quad P(a^*)=\prod_{j=1}^d e^{-p^*_j/2}=1.
\end{equation}
Let us show that $a^*$ is unique. If $(L_1(a^*),\dots, L_d(a^*))=(L_1(a),\dots, L_d(a))$ for some $a \in (0,+\infty)^d$ with $P(a)=1$ then setting $p_i=-2\log(a_i)$ we have that both $p,p^* \in H$ and
\begin{equation}
\nabla \Phi_0(p)=(d_1,\dots,d_d)=\nabla \Phi_0(p^*).
\end{equation}
Hence, the strict convexity of $\Phi_D$ in $H$ implies that $p=p^*$ thus  $a=a^*$.
Indeed, let us argue by contradiction supposing that  the map $g(t):=\Phi_0(p+t(p^*-p))$ defined on $[0,1]$ is not constant. Then $g$ takes values into $H$ and  satisfies 
\begin{equation}
g''(t):=(p^*-p)\cdot D^2\Phi_0(p+t(p^*-p)) (p^*-p) >0.
\end{equation}
Hence, $g'$ is strictly increasing while 
\begin{equation}
g'(0)=\nabla \Phi_0(p)=\nabla \Phi_0(p^*)=g'(1),
\end{equation}
a contradiction.

 \noindent For any $r>0$, by  \eqref{proof_lemma_elipsoid_dge3_3}
\begin{equation}
(L_1(ra^*),\dots, L_d(ra^*))=(d_1,\dots,d_d)
\end{equation}
and imposing the volume condition 
\begin{equation}
1=|E(ra^*)|=\omega_dr^dP(a)=\omega_dr^d
\end{equation}
we obtain $r=\omega_d^{-1/d}$. In conclusion, we have completed the proof in the case $B$ is a diagonal matrix. 

To deal with the general case, let $R$ be as in \eqref{def_BRD} so that $B=RDR^T$ and let $E_D$ be the unique ellipsoid obtained for the diagonal matrix $D$
of vouìlume $1$.
Then, by \eqref{def_NE} and the radiality of $\Gamma$,
\begin{multline}
N_{RE_D}(x)=\int_{RE_D}\Gamma(x-y) \, dy=\int_{E_D}\Gamma(x-Ry) \, dy=\int_{E_D}\Gamma(R^Tx-y) \, dy=N_{E_D}(R^Tx).
\end{multline}
Differentiating on $RE_D$
\begin{equation}
\nabla N_{RE_D}(x)=R \nabla N_{E_D}(R^Tx)=-RDR^Tx=-Bx \quad \text{ in } RE_D,
\end{equation}
thus completing the proof.
\end{proof}

\begin{lemma}\label{lemma_elipsoid_d2}
Let $k\ge d=2$ and let $A$ be a $k\times 2$ matrix of rank $2$. Let $S:=(A^TA)^{1/2}$, $B:=(\text{\rm Tr}(S))^{-1}S$, $D$ and $R$ be as in \eqref{def_BRD}.
Then, the unique ellipsoid $E_A$ centered in $0$ with volume $1$ such that
\begin{equation}\label{eq_nabla_NEa_d2}
\nabla N_{E_A}(x)=-Bx \quad \text{ in } E_A
\end{equation}
is  given by,  
\begin{equation}\label{def_EA_d2}
E_A=R E(\pi^{-1/2}a^*),
\end{equation}
with
\begin{equation}
 a^*:=\left(\sqrt{\frac{d_2}{d_1}},\sqrt{\frac{d_1}{d_2}}\right),
\end{equation}
where $\{d_1,d_2\}$ are the entries on the diagonal $2\times 2$ matrix $D$.
\end{lemma}

\begin{proof}
Suppose at first $B$ is a diagonal matrix, where $B$ is as in \eqref{def_BRD}. In view of \cite[Proposition 3.1]{CMMRSV_d2} with $\alpha=0$, for any ellipse  $E(a)$ with $a \in (0,+\infty)^2$
\begin{equation}
\nabla N_{E(a)}(x)=-\left(\frac{a_2 x_1}{a_1+a_2}, \frac{a_1 x_2}{a_1+a_2}\right)  \quad \text{ in } E(a).
\end{equation}
Imposing the conditions 
\begin{equation}
d_1=\frac{a_2}{a_1+a_2}\qquad\text{and}\qquad 
d_2=\frac{a_1}{a_1+a_2},
\end{equation}
we obtain 
\begin{equation}
\frac{a_1}{a_2}=\frac{d_2}{d_1}.
\end{equation}
Furthermore,
\begin{equation}
1=|E(a)|=\pi a_1a_2.
\end{equation}
Thus, we must have 
\begin{equation}
\pi^{1/2} a=\left(\sqrt{\frac{d_2}{ d_1}},\sqrt{\frac{d_1}{ d_2}}\right).
\end{equation}
The general case follows arguing as in Lemma \ref{lemma_elipsoid_dge3}.
\end{proof}


\section{The value of $\La^*(A)$}\label{subsec_La}
In this section we prove Theorem \ref{theo_main_1} and we show that the function $V_A$ from Theorem \ref{theo_main_2} is a minimizer of \eqref{prob_min}; we will complete the proof of Theorem \ref{theo_main_2} in the next section, where we will prove that $V_A$ is the unique minimizer of \eqref{prob_min}. 
We start with the following lemma, which allows to replace the measure constraint in \eqref{prob_min} by an inequality. 

\begin{lemma}\label{lemma_ge_1}
Let $d\ge 2$, $k\ge d$ and $A$ be a $k\times d$ matrix of rank $d$.
Then,
\begin{equation}\label{prob_min_ge_1}
\Lambda^\ast(A)=\inf\left\{\int_{\R^d} |\nabla V|^2\, dx:V\in D^{1,2}(\R^d, \R^k),  |\{V(x)= Ax\}|\ge1\right\}
\end{equation}
and any minimizer $V$ of \eqref{prob_min_ge_1} minimizes \eqref{prob_min}, that is,  $|\{V(x)= Ax\}|=1$.
\end{lemma}

\begin{proof}
We follow the proof of our \cite[Proposition 4.1]{SV_blowups}.
Let $V\in D^{1,2}(\R^d, \R^k)$. We set 
$$m:=|\{V(x)= Ax\}|\ge1,$$ 
and we define the rescaling
\begin{equation}
V_m(x):=\frac{1}{m^{1/d}}V(m^{1/d}x).
\end{equation}
It follows that
\begin{equation}
|\{V_m(x)= Ax\}|=m^{-1}|\{V(x)= Ax\}|=1
\end{equation}
and by a change of variables 
\begin{equation}\label{proof_lemma_ge_1_1}
\int_{\R^d}|\nabla V_m|^2 \,dx=\frac{1}{m}\int_{\R^d}|\nabla V|^2 \,dx.
\end{equation}
Hence, 
\begin{multline}
\inf\left\{\int_{\R^d} |\nabla V|^2\, dx:V\in D^{1,2}(\R^d, \R^k),  |\{V(x)= Ax\}|\ge1\right\}\\
=\inf_{m\ge 1}\inf\left\{\int_{\R^d} |\nabla V|^2\, dx:V\in D^{1,2}(\R^d, \R^k),  |\{V(x)= Ax\}|=m\right\}\\
=\inf_{m\ge 1}\inf\left\{m\int_{\R^d} |\nabla V|^2\, dx:V\in D^{1,2}(\R^d, \R^k),  |\{V(x)= Ax\}|=1\right\}\\
=\inf\left\{\int_{\R^d} |\nabla V|^2\, dx:V\in D^{1,2}(\R^d, \R^k),  |\{V(x)= Ax\}|=1\right\}=\Lambda^*(A),
\end{multline}
thus we have proved \eqref{prob_min_ge_1}. %
Furthermore, if $V$ is a minimizer of \eqref{prob_min_ge_1} then, in view of \eqref{proof_lemma_ge_1_1}, we must have  $|\{V(x)= Ax\}|=1$.
Thus, any minimizer $V$ of \eqref{prob_min_ge_1} is a minimizer of \eqref{prob_min} as well.
\end{proof}

In order to prove Theorem \ref{theo_main_1}, we first find a lower bound for $\Lambda^*(A)$ and then we build a function $V_A$ that realizes this lower bound. 
The main result of the section is the following.

\begin{proposition}\label{prop_LA}
Let $d \ge 2$, $k \ge d$,  $A \in \R^{k,d}$ be a $k \times d$ matrix with ${\rm{rk}}(A)=d$.  Then, the following holds: 
\begin{enumerate}[\quad\rm(i)]
\item For every $V\in D^{1,2}(\R^d, \R^k)$ such that $|\{V(x)= Ax\}|\ge1$ we have 
\begin{equation}\label{eq_LA_lower_bound}
\int_{\R^d} |\nabla V|^2\, dx\ge(\Tr(\sqrt{A^TA}))^2.
\end{equation}
\item Consider the function 
\begin{equation}
V_A:\R^k\to\R^d\ ,\qquad V_A:=-(\Tr(\sqrt{A^TA}))U \nabla N_{E_A},
\end{equation}
where $U$ is the $k\times d$ matrix from \eqref{def_SU} and $N_{E_A}$ is the solution from  \eqref{def_NE} with $E_A$ given by \eqref{def_EA_dge3} or \eqref{def_EA_d2} if $d \ge 3$ or $d=2$, respectively. Then, 
$V_A\in D^{1,2}(\R^d, \R^k)$, $|\{V_A(x)= Ax\}|=1$ and 
\begin{equation}\label{eq_LA_upper_bound}
\int_{\R^d} |\nabla V_A|^2\, dx=(\Tr(\sqrt{A^TA}))^2.
\end{equation}
\end{enumerate}
In particular, {\rm(i)} and {\rm(ii)} imply that 
\begin{equation}\label{eq_LA}
\La^*(A)=(\Tr(\sqrt{A^TA}))^2
\end{equation}
and that 
$V_A$ is a minimizer of \eqref{prob_min} and \eqref{prob_min_ge_1}.
\end{proposition}
\begin{proof} We start with the proof of claim (i). 
Let $U$ be the $k\times d$ matrix from \eqref{def_SU} and let 
$$W(x):=U^T V(x).$$ Then $W$ takes values in $\R^d$ and 
\begin{equation}\label{proof_lemma_lower_bound_0}
W(x)=U^T V(x)=U^TA x=Sx \quad \text{ a.e. on } K,
\end{equation}
where $S$ is as in \eqref{def_SU}. In  particular
\begin{equation}\label{proof_lemma_lower_bound_0_5}
\nabla W(x)=S \quad \text{ a.e. on } K.
\end{equation}
It follows that the divergence of $W$ is constant on $K$ and more precisely
\begin{equation}\label{proof_lemma_lower_bound_1}
\dive(W)=\Tr(S) \quad \text{ a.e. on } K.
\end{equation}
We notice that $W \in H^1_{loc}(\R^d,\R^d)$ and $\nabla W \in  L^2(\R^d,\R^{d,d})$. Thus, taking the Fourier transform and by the Cauchy-Schwarz inequality, we get
\begin{equation}\label{proof_lemma_lower_bound_2}
\int_{\R^d}  |\dive(W)|^2 \, dx=\int_{\R^d}  |\widehat{\dive(W)}|^2 \, d\xi =\int_{\R^d}  |\xi \cdot \widehat W|^2 \, d\xi
\le\int_{\R^d}  |\xi|^2  |\widehat W|^2 \, d\xi= \int_{\R^d}  |\nabla W|^2 \, dx.
\end{equation}
Thanks to  \eqref{proof_lemma_lower_bound_1}, \eqref{proof_lemma_lower_bound_2}, and the constraint $|K|\ge 1$, we get
\begin{equation}\label{proof_lemma_lower_bound_3}
\begin{array}{rl}
\displaystyle\Tr(S)^2\le \int_K  \displaystyle\Tr(S)^2 \, dx =\int_K  |\dive(W)|^2 \, dx&\displaystyle \le \int_{\R^d}  |\dive(W)|^2 \, dx\smallskip\\
&\displaystyle\le  \int_{\R^d}  |\nabla W|^2 \, dx
=\int_{\R^d}  |U^T\nabla V|^2 \, dx \le \int_{\R^d}  |\nabla V|^2 \, dx,
\end{array}
\end{equation}
which concludes the proof of \eqref{eq_LA_lower_bound}.\medskip

\noindent{We next prove claim {\rm(ii)}.} In view of Lemmas \ref{lemma_elipsoid_dge3} and \ref{lemma_elipsoid_d2}, we have 
$$ \nabla N_{E_A}(x)=-Bx \quad  \text{ for } x\in E_A.$$
Thus, using the definition of $V_A$ and the identities \eqref{def_SU} and \eqref{def_BRD}, we get
\begin{equation}
V_A(x)=-(\Tr(\sqrt{A^TA}))U \nabla N_{E_A}(x)=(\Tr(\sqrt{A^TA}))UBx=USx=Ax \quad  \text{ in } E_A.
\end{equation}
In particular, we obtain the following estimate for the measure of the contact set: 
\begin{equation}
|\{V_A(x)=Ax\}| \ge |E_A|=1.
\end{equation}
We next prove \eqref{eq_LA_upper_bound}. By \eqref{def_SU}, \eqref{def_NE} and the Plancherel identity for the Fourier transform
\begin{align}
\label{proof_prop_LA_1}
\int_{\R^d} |\nabla V_A|^2 \, dx&=(\Tr(\sqrt{A^TA}))^2\int_{\R^d} |D^2N_{E_A}|^2 \, dx\\
&=(\Tr(\sqrt{A^TA}))^2\int_{\R^d} \sum_{i,j=1}^d\xi_j^2\xi_i^2|\widehat N_{E_A}|^2 \, d\xi\notag\\
&=(\Tr(\sqrt{A^TA}))^2\int_{\R^d} \left(\sum_{i=1}^d\xi_i^2\right)^2|\widehat N_{E_A}|^2 \, d\xi\notag\\
&=(\Tr(\sqrt{A^TA}))^2\int_{\R^d} |\Delta N_{E_A}|^2 \, dx\notag\\
&=(\Tr(\sqrt{A^TA}))^2\int_{\R^d} |\chi_{E_A}|^2 \, dx=(\Tr(\sqrt{A^TA}))^2.\notag
\end{align}
In particular, $\nabla V_A \in L^2(\R^d,\R^{k,d})$ for $d \ge 2$ and, by the Sobolev inequality,  $V_A \in D^{1,2}(\R^d,\R^k)$ for $d \ge3$.
If $d=2$ let us check that $V_A \in L^{2,\infty}(\R^d,\R^k)$. By \eqref{def_Gamma} and \eqref{def_NE}
\begin{equation}
\nabla N_{E_A}(x)=-\frac{1}{2 \pi}\int_{E_A} \frac{x-y}{|x-y|^2} \, dy.
\end{equation}
Let $R>0$ be such that $E_A \subset B_R$. Then on $\R^2 \setminus B_{2R}$ we have $|x-y| \ge |x|-|y| \ge |x|/2$ for any $y \in E_A$, thus
\begin{equation}
|\nabla N_{E_A}(x)| \le \frac{1}{2 \pi}\int_{E_A} \frac{1}{|x-y|} \, dy \le \frac{|E_A|}{\pi |x|}=\frac{1}{\pi |x|}.
\end{equation}
On the other hand on $B_{2R}$
\begin{equation}
|\nabla N_{E_A}(x)| \le \frac{1}{2 \pi}\int_{E_A} \frac{1}{|x-y|} \, dy=\frac{1}{2 \pi}\int_{E_A-x} \frac{1}{|z|} \, dz 
\le \frac{1}{2 \pi}\int_{B_{3R}} \frac{1}{|z|} \, dz =\frac{3}{2}R.
\end{equation}
In conclusion, for some constant $C>0$ depending only on $R$ and $A$,
\begin{equation}
|V_A(x)| \le \Tr(\sqrt{A^TA})|\nabla N_{E_A}(x)| \le \frac{C}{1+|x|} \quad \text{ in } \R^2.
\end{equation}
It follows that if $0<\la\le C$
\begin{equation}
\{|V_A| > \la\} \subset B_{(C-\la)/\la} \subset  B_{C/\la},
\end{equation}
while $\{|V_A| > \la\}=\emptyset$ if $\la \ge C$.
Hence, for any $\la >0$
\begin{equation}
\la |\{|V_A| > \la\}|^{1/2} \le C \sqrt{\pi},
\end{equation}
thus showing that $V_A \in L^{2,\infty}(\R^2,\R^k)$ in view of \eqref{def_quasinorm_L2infty} and \eqref{def_L2infty}. This concludes the proof of (ii).\medskip 

\noindent{\it Conclusion.} By (ii), we know that $V_A$ is an admissible competitor in \eqref{prob_min_ge_1}. Thus, combining the bounds \eqref{eq_LA_lower_bound} and \eqref{eq_LA_upper_bound}, we get that $V_A$ minimizes \eqref{prob_min_ge_1}. Finally, thanks to Lemma \ref{lemma_ge_1}, we have that $|\{V_A=Ax\}|=1$ and that \eqref{eq_LA} holds. 
\end{proof}

\section{Rigidity}\label{sec_rig}

In this section we complete the proof of Theorem \ref{theo_main_2} by showing that any minimizer must have an ellipsoidal coincidence set. 
The Hopf-type argument used in the proof was developed in \cite[Proof of the main theorem, step 4]{EW_short_proof} in the context of the obstacle problem.

\begin{proof}[Proof of Theorem \ref{theo_main_2}]
We have already proved \eqref{eq_La_theo_main} in Proposition \ref{prop_LA}.
Let $V$ be a minimizer of \eqref{prob_min} and $K:=\{V(x)=Ax\}$ its coincidence set, in particular $|K|=1$ and, by \cite[Theorem 1.6]{SV_blowups},
$K$ is bounded.
In view of Proposition \ref{prop_LA} 
\begin{equation}\label{proof_theo_main_1}
\int_{\R^d} |\nabla V|^2 \, dx =(\Tr(\sqrt{A^TA}))^2.
\end{equation}
As in Proposition \ref{prop_LA}, let us define $W:=U^TV$ with $U$ as in \eqref{def_SU}. 
Since $V \in D^{1,2}(\R^d, \R^k)$, it follows that $W \in D^{1,2}(\R^d, \R^d)$ by \eqref{def_SU}.

Furthermore, we recall that, by \eqref{proof_lemma_lower_bound_0}, \eqref{proof_lemma_lower_bound_0_5} and \eqref{proof_lemma_lower_bound_1},
with $S$ as in \eqref{def_SU},
\begin{equation}\label{proof_theo_main_2}
W(x)=Sx, \quad \nabla W(x)=S  \quad \text{ and } \quad  \dive(W)=\Tr(S) \quad \text{ a.e. in } K
\end{equation}
and that by \eqref{eq_La_theo_main}, \eqref{def_SU}, \eqref{proof_lemma_lower_bound_2} and the minimality of $V$
\begin{multline}\label{proof_theo_main_3}
\Lambda^*(A)=\Tr(S)^2= \int_K  |\dive(W)|^2 \, dx \le \int_{\R^d}  |\dive(W)|^2 \, dx \le  \int_{\R^d}  |\nabla W|^2 \, dx\\
=\int_{\R^d}  |U^T\nabla V|^2 \, dx \le \int_{\R^d}  |\nabla V|^2 \, dx=\Lambda^*(A).
\end{multline}
Hence, all the inequalities in \eqref{proof_theo_main_3} are actually equalities. In particular we must have, by \eqref{proof_theo_main_2},
\begin{equation}
\dive(W)=\Tr(S)\chi_K \quad \text{a.e. in } \R^d.
\end{equation}
Moreover, 
\begin{equation}\label{proof_theo_main_4}
\int_{\R^d}|\nabla W|^2\,dx =\int_{\R^d}|\dive(W)|^2+\frac{1}{2}\sum_{i,j=1}^d \left(\pd{W_j}{x_i}-\pd{W_i}{x_j}\right)^2 \, dx.
\end{equation}
Indeed, it is enough to check \eqref{proof_theo_main_4} for any $\varphi \in C^{\infty}_c(\R^d,\R^k)$ and, integrating by parts,
\begin{align*}
\frac{1}{2}\sum_{i,j=1}^d\int_{\R^d} \left(\pd{\varphi_j}{x_i}-\pd{\varphi_i}{x_j}\right)^2 \, dx
&=\frac{1}{2}\sum_{i,j=1}^d\int_{\R^d} \left(\left|\pd{\varphi_j}{x_i}\right|^2+\left|\pd{\varphi_i}{x_j}\right|^2
-2\pd{\varphi_j}{x_i}\pd{\varphi_i}{x_j}\right) \, dx\\
&=\int_{\R^d}|\nabla \varphi|^2 \, dx 
-\sum_{i,j=1}^d\int_{\R^d}\varphi_j\pd{^2\varphi_i}{x_jx_i} \, dx\\
&=\int_{\R^d}|\nabla \varphi|^2 \, dx 
-\sum_{i=1}^d\int_{\R^d}\varphi_j\pd{}{x_j}(\dive(\varphi)) \, dx\\
&=\int_{\R^d}|\nabla \varphi|^2 \, dx 
-\int_{\R^d}|\dive(\varphi)|^2 \, dx.
\end{align*}
By \eqref{proof_theo_main_3} and \eqref{proof_theo_main_4}, we conclude that 
\begin{equation}\label{proof_theo_main_5}
\pd{W_j}{x_i}=\pd{W_i}{x_j} \quad \text{ for any } i,j=1,\dots,d.
\end{equation}
Let $N_K$ be as in \eqref{def_NE} for $E=K$ and let us define $H:=W+\Tr(S) \nabla N_K$. Then by \eqref{eq_NE}
\begin{equation}\label{proof_theo_main_6}
\dive(H)=\dive(W)+\dive(\Tr(S)\nabla N_K)=\Tr(S)[\chi_K+\Delta N_K]=0.
\end{equation}
Furthermore, by \eqref{proof_theo_main_5},
\begin{equation}
\pd{H_j}{x_i}=\pd{H_i}{x_j} \quad \text{ for any } i,j=1,\dots,d
\end{equation}
and $\nabla H \in L^2(\R^d,\R^d)$ since, as in \eqref{proof_prop_LA_1},
\begin{multline}
\int_{\R^d} |\nabla H|^2\, dx \le 2\int_{\R^d} |\nabla W|^2\, dx +2\Tr(S)^2\int_{\R^d} |D^2 N_K|^2\, dx \\
\le 2\int_{\R^d} |\nabla W|^2\, dx +2\Tr(S)^2\int_{\R^d} |\Delta N_K|^2\, dx
=2\int_{\R^d} |\nabla W|^2\, dx +2\Tr(S)^2.
\end{multline}
Finally, by \eqref{proof_theo_main_4}, \eqref{proof_theo_main_5} and \eqref{proof_theo_main_6}
\begin{equation}\label{proof_theo_main_7}
\int_{\R^d} |\nabla H|^2\, dx=\int_{\R^d}|\dive(H)|^2+\frac{1}{2}\sum_{i,j=1}^d \left(\pd{H_j}{x_i}-\pd{H_i}{x_j}\right)^2 \, dx=0.
\end{equation}
It follows that for some $c \in \R^d$
\begin{equation}\label{proof_theo_main_7_5}
W=-\Tr(S) \nabla N_K +c \quad \text{a.e. in } \R^d.
\end{equation}
In dimension $d \ge 3$,  $\nabla N_K  \in D^{1,2}(\R^d,\R^d)$  by the Sobolev inequality into account. Hence, since also $W \in D^{1,2}(\R^d,\R^d)$, we must have $c=0$.
Arguing as in the proof of Proposition \ref{prop_LA}, since $K$ is bounded as already observed,  $\nabla N_K  \in D^{1,2}(\R^2,\R^2)$.
Hence,  we must have $c=0$ also for $d=2$.

In particular,
\begin{equation}\label{proof_theo_main_8}
\nabla N_K(x)=-\frac{W(x)}{\Tr(S)}=-\frac{U^T Ax}{\Tr(S)}=-\frac{Sx}{\Tr(S)}=-Bx \quad \text{ a.e. in } K.
\end{equation}
We observe that $\nabla N_K(x)$ is continuous. Indeed, by \eqref{def_Gamma} and \eqref{def_NE},
\begin{equation}
\nabla N_K(x)=-\frac{1}{d \omega_d}\int_K\frac{x-y}{|x-y|^d} \, dy=\frac{1}{d \omega_d}\int_{K-x}\frac{z}{|z|^{d}} \, dz,
\end{equation}
$K$ is bounded, and $\frac{z}{|z|^{d}}$ is locally integrable. Hence, \eqref{proof_theo_main_8} holds in a pointwise sense.

Let $E_A$ be as in \eqref{def_EA_dge3} or \eqref{def_EA_d2}.
We notice that,  by a change of variables, \eqref{def_Gamma} and \eqref{def_NE},
\begin{equation}
\nabla N_{tE_A}(x)=-\int_{tE_A}\frac{x-y}{|x-y|^d}\, dy=-t^{d}\int_{E_A}\frac{x-ty}{|x-ty|^d} \, dz=t\nabla N_{E_A}(x/t)
\end{equation}
thus in particular, by \eqref{eq_nabla_NEa_dge3} or \eqref{eq_nabla_NEa_d2}
\begin{equation}\label{proof_theo_main_9}
\nabla N_{tE_A}(x)=-Bx \quad \text{ in } E_{tA}.
\end{equation}
Furthermore, by  \eqref{def_EA_dge3} or \eqref{def_EA_d2}, for some $f_i \in \R$ for $i=1,\dots,d$ and $R$ as in \eqref{def_BRD}
\begin{align*}
E_A=R E(f)&=R\left(\left\{x \in \R^d:\sum_{i=1}^d\frac{x_i^2}{f_i^2}<1\right\}\right)\\
&=R(\{x \in \R^d:x \cdot Fx<1\})=\{x \in \R^d:x \cdot RFR^Tx<1\},
\end{align*}
where $F$ is the diagonal matrix with entries $\{f_i^{-2}\}_{i=1,\dots,d}$.
Let us define $G:= RFR^T$ and 
\begin{equation}
t^*:=\max_{x \in K}\sqrt{x\cdot G x}
\end{equation}
thus, since $\overline{t^* E_A}=\{x \in \R^d:x \cdot Gx\le(t^*)^2\}$,
\begin{equation}
K \subset \overline{t^* E_A} \quad \text{ and }\quad  K\cap  \partial (t^* E_A)\neq \emptyset.
\end{equation}
Let $x_0 \in K\cap  \partial (t^* E_A)$.
In view of \eqref{proof_theo_main_8} and \eqref{proof_theo_main_9}
\begin{equation}
\nabla N_K(x_0)=-Bx_0=\nabla N_{t^* E_A}(x_0),
\end{equation}
or equivalently, by \eqref{def_NE},
\begin{equation}\label{proof_theo_main_10}
\nabla N_{(t^* E_A)\setminus K}(x_0)=0.
\end{equation}
For any $y \in t^*E_A$, the Cauchy-Schwarz inequality with respect to the scalar product $ Gx\cdot y$ in $\R^d$ yields
\begin{equation}
G x_0 \cdot y \le \sqrt{x_0 \cdot G x_0}\sqrt{  y\cdot G y}<(t^*)^2 =x_0\cdot G x_0,
\end{equation}
thus 
\begin{equation}\label{proof_theo_main_11}
G x_0 \cdot (x_0-y)>0 \quad \text{ in } t^*E_A.
\end{equation}
Hence, by \eqref{proof_theo_main_11},
\begin{equation}
G x_0\cdot \nabla N_{(t^* E_A)\setminus K}(x_0)=-\frac{1}{d\omega_d}\int_{(t^*E_A)\setminus K}\frac{Gx_0\cdot (x_0-y)}{|x_0-y|^d}\, dy\le0
\end{equation}
with a strict inequality if $|(t^*E_A)\setminus K| \neq 0$. By \eqref{proof_theo_main_10} the inequality cannot be strict, thus we conclude that  $|(t^*E_A)\setminus K| =0$.
Furthermore, 
\begin{equation}
1=|K|=|t^*E_A|=(t^*)^d
\end{equation}
which implies $t^*=1$ and so $K=E_A$ up to a set of volume $0$. 
In conclusion, we have shown that 
\begin{equation}\label{proof_theo_main_11_5}
W=-\Tr(S) \nabla N_{E_A}\quad \text{a.e. in } \R^d,
\end{equation}
with $E_A$ as in \eqref{def_EA}. 

Next, we show that $V=UW$. To this end, we observe that the map $P:=UU^T$ is the orthogonal projection onto ${\rm Im} (U)$. Indeed, 
\begin{equation}
(x-UU^Tx)\cdot y=(x-UU^Tx)\cdot Uz=U^T(x-UU^Tx)\cdot z=0
\end{equation}
for any $y \in {\rm Im} (U)$, since $U^TU={\rm Id}_d$. Hence, 
\begin{equation}
|\nabla V|^2=|P \nabla V|^2 +|({\rm Id}_k-P)\nabla V|^2=|U^T\nabla V|^2 +|({\rm Id}_k-P)\nabla V|^2
\end{equation}
and \eqref{proof_theo_main_3} yields $({\rm Id}_k-P)\nabla V=0$. If follows that for some $c \in ({\rm Im} (U))^{\perp}$
\begin{equation}
({\rm Id}_k-P)V=c.
\end{equation}
Since on $K$, a set of volume $1$,
\begin{equation}
V(x)=Ax=USx \in {\rm Im} (U),
\end{equation}
we must have $c=0$. In conclusion,
\begin{equation}
V=PV=UU^TV=UW.
\end{equation}
Hence, we have proved \eqref{eq_Va_theo_main} in view of \eqref{proof_theo_main_11_5}. All that it is left is to prove \eqref{eq_concidence_set_theo_main}.
To this end, let us  show that
\begin{equation}\label{proof_theo_main_12}
\nabla N_{E_A}(x) \neq -Bx \quad \text{ for any } x \not \in \overline{E_A}. 
\end{equation}
Let argue by contradiction supposing that $\nabla N_{E_A}(x)=-Bx$ for some $x \not \in \overline{E_A}$ and let
\begin{equation}
t:=\sqrt{x\cdot Gx}>1.
\end{equation}
Then $x \in \partial (t E_A)$ so that $\nabla N_{t E_A}(x)=-Bx$ and  
\begin{equation}
0=\nabla N_{E_A}(x)+Bx=\nabla N_{E_A}(x)-\nabla N_{t E_A}(x)=-\nabla N_{(tE_A)\setminus E_A}(x).
\end{equation}
Arguing as done above for $t^*E_A$ and $K$, we conclude that $|(t E_A)\setminus E_A|=0$, a contradiction to $t>1$.
In conclusion, since $U$ is injective and $V_A(x)-Ax=-\Tr(S)U(\nabla N_{E_A}(x)+Bx)$,
\begin{equation}
V_A(x)= Ax \quad \text{ if and  only if }\quad \nabla N_{E_A}(x)=-Bx,  
\end{equation}
which happens exactly in  $\overline{E_A}$ in view of  \eqref{eq_nabla_NEa_dge3}, \eqref{eq_nabla_NEa_d2} and \eqref{proof_theo_main_12}. 
The proof is thereby complete.
\end{proof}

\section*{Acknowledgements}
\noindent
G. Siclari is supported by ``Centro di Ricerca Matematica Ennio De Giorgi''.
G. Siclari is partially supported by the 2026 INdAM--GNAMPA project 2026 ``Asymptotic analysis of variational problems''.
B. Velichkov was supported by the European Research Council's (ERC) project ERC VaReg - \it Variational approach to the regularity of the free boundaries \rm and the project  ERC FiRM - \it Fine structure and regularity of moving and stationary free interfaces\rm, and the MIUR Excellence Department Project awarded to the Department of Mathematics (CUP I57G22000700001).

\bibliographystyle{acm}
\bibliography{references}	

@article{SV_blowups,
      title={On the blow-up of the vectorial Bernoulli free boundary problem}, 
      author={Giovanni Siclari and Bozhidar Velichkov},
      journal={arXiv preprint arXiv:2602.00741},
      year={2026},
      eprint={2602.00741},
      archivePrefix={arXiv},
      primaryClass={math.AP},
      url={https://arxiv.org/abs/2602.00741}, 
}

@article {CMMRSV_d2,
    AUTHOR = {Carrillo, J. A. and Mateu, J. and Mora, M. G. and Rondi, L.
              and Scardia, L. and Verdera, J.},
     TITLE = {The ellipse law: {K}irchhoff meets dislocations},
   JOURNAL = {Comm. Math. Phys.},
  FJOURNAL = {Communications in Mathematical Physics},
    VOLUME = {373},
      YEAR = {2020},
    NUMBER = {2},
     PAGES = {507--524},
      ISSN = {0010-3616,1432-0916},
   MRCLASS = {31B15 (31B35 35B35 35K59 49Q10 49Q20 74G65)},
  MRNUMBER = {4056642},
MRREVIEWER = {Angelica\ Malaspina},
       DOI = {10.1007/s00220-019-03368-w},
       URL = {https://doi.org/10.1007/s00220-019-03368-w},
}

@article {EFW_glob_obs,
    AUTHOR = {Eberle, Simon and Figalli, Alessio and Weiss, Georg S.},
     TITLE = {Complete classification of global solutions to the obstacle
              problem},
   JOURNAL = {Ann. of Math. (2)},
  FJOURNAL = {Annals of Mathematics. Second Series},
    VOLUME = {201},
      YEAR = {2025},
    NUMBER = {1},
     PAGES = {167--224},
      ISSN = {0003-486X,1939-8980},
   MRCLASS = {35R35 (31B05 31B20 35J86)},
  MRNUMBER = {4848670},
MRREVIEWER = {Mariana\ Smit Vega Garcia},
       DOI = {10.4007/annals.2025.201.1.3},
       URL = {https://doi.org/10.4007/annals.2025.201.1.3},
}

@article {D_att,
    AUTHOR = {Dive, Pierre},
     TITLE = {Attraction des ellipso\"ides homog\`enes et r\'eciproques d'un
              th\'eor\`eme de {N}ewton},
   JOURNAL = {Bull. Soc. Math. France},
  FJOURNAL = {Bulletin de la Soci\'et\'e{} Math\'ematique de France},
    VOLUME = {59},
      YEAR = {1931},
     PAGES = {128--140},
      ISSN = {0037-9484},
   MRCLASS = {99-04},
  MRNUMBER = {1504977},
       URL = {http://www.numdam.org/item?id=BSMF_1931__59__128_0},
}

@book {PSU_conv,
    AUTHOR = {Petrosyan, Arshak and Shahgholian, Henrik and Uraltseva, Nina},
     TITLE = {Regularity of free boundaries in obstacle-type problems},
    SERIES = {Graduate Studies in Mathematics},
    VOLUME = {136},
 PUBLISHER = {American Mathematical Society, Providence, RI},
      YEAR = {2012},
     PAGES = {x+221},
      ISBN = {978-0-8218-8794-3},
   MRCLASS = {35R35 (35B65 35Q91)},
  MRNUMBER = {2962060},
MRREVIEWER = {Michele\ Miranda},
       DOI = {10.1090/gsm/136},
       URL = {https://doi.org/10.1090/gsm/136},
}

@article {L_inv,
    AUTHOR = {Lewy, Hans},
     TITLE = {An inversion of the obstacle problem and its explicit
              solution},
   JOURNAL = {Ann. Scuola Norm. Sup. Pisa Cl. Sci. (4)},
  FJOURNAL = {Annali della Scuola Normale Superiore di Pisa. Classe di
              Scienze. Serie IV},
    VOLUME = {6},
      YEAR = {1979},
    NUMBER = {4},
     PAGES = {561--571},
      ISSN = {0391-173X,2036-2145},
   MRCLASS = {35R30 (35J05)},
  MRNUMBER = {563334},
MRREVIEWER = {S.\ D\"ummel},
       URL = {http://www.numdam.org/item?id=ASNSP_1979_4_6_4_561_0},
}

@article {FS_quad,
    AUTHOR = {Friedman, Avner and Sakai, Makoto},
     TITLE = {A characterization of null quadrature domains in {${\bf
              R}^N$}},
   JOURNAL = {Indiana Univ. Math. J.},
  FJOURNAL = {Indiana University Mathematics Journal},
    VOLUME = {35},
      YEAR = {1986},
    NUMBER = {3},
     PAGES = {607--610},
      ISSN = {0022-2518,1943-5258},
   MRCLASS = {31B05},
  MRNUMBER = {855176},
MRREVIEWER = {D.\ H.\ Armitage},
       DOI = {10.1512/iumj.1986.35.35031},
       URL = {https://doi.org/10.1512/iumj.1986.35.35031},
}

@article {DF_bubble,
    AUTHOR = {DiBenedetto, Emmanuele and Friedman, Avner},
     TITLE = {Bubble growth in porous media},
   JOURNAL = {Indiana Univ. Math. J.},
  FJOURNAL = {Indiana University Mathematics Journal},
    VOLUME = {35},
      YEAR = {1986},
    NUMBER = {3},
     PAGES = {573--606},
      ISSN = {0022-2518,1943-5258},
   MRCLASS = {76S05 (35R35 76V05)},
  MRNUMBER = {855175},
MRREVIEWER = {Michel\ Chipot},
       DOI = {10.1512/iumj.1986.35.35030},
       URL = {https://doi.org/10.1512/iumj.1986.35.35030},
}

@article {S_dim2,
    AUTHOR = {Sakai, Makoto},
     TITLE = {Null quadrature domains},
   JOURNAL = {J. Analyse Math.},
  FJOURNAL = {Journal d'Analyse Math\'ematique},
    VOLUME = {40},
      YEAR = {1981},
     PAGES = {144--154},
      ISSN = {0021-7670,1565-8538},
   MRCLASS = {30E99 (31A05)},
  MRNUMBER = {659788},
MRREVIEWER = {L.\ I.\ Hedberg},
       DOI = {10.1007/BF02790159},
       URL = {https://doi.org/10.1007/BF02790159},
}

@article {SW_glob_dim2,
    AUTHOR = {Salib, Anthony and Weiss, Georg S.},
     TITLE = {Classification of global solutions to the obstacle problem in
              the plane},
   JOURNAL = {Adv. Math.},
  FJOURNAL = {Advances in Mathematics},
    VOLUME = {472},
      YEAR = {2025},
     PAGES = {Paper No. 110276, 36},
      ISSN = {0001-8708,1090-2082},
   MRCLASS = {35B08 (35J05 35R35)},
  MRNUMBER = {4892696},
MRREVIEWER = {Zixiao\ Liu},
       DOI = {10.1016/j.aim.2025.110276},
       URL = {https://doi.org/10.1016/j.aim.2025.110276},
}

@article {MTV_reg_vect,
    AUTHOR = {Mazzoleni, Dario and Terracini, Susanna and Velichkov,
              Bozhidar},
     TITLE = {Regularity of the free boundary for the vectorial {B}ernoulli
              problem},
   JOURNAL = {Anal. PDE},
  FJOURNAL = {Analysis \& PDE},
    VOLUME = {13},
      YEAR = {2020},
    NUMBER = {3},
     PAGES = {741--764},
      ISSN = {2157-5045,1948-206X},
   MRCLASS = {35R35 (35J60 49K10 49Q20)},
  MRNUMBER = {4085121},
MRREVIEWER = {Leandro\ S.\ Tavares},
       DOI = {10.2140/apde.2020.13.741},
       URL = {https://doi.org/10.2140/apde.2020.13.741},
}

@article {TV_survey,
    AUTHOR = {Tortone, Giorgio and Velichkov, Bozhidar},
     TITLE = {Vectorial {B}ernoulli problems and free boundary systems},
   JOURNAL = {Matematica},
  FJOURNAL = {La Matematica. Official Journal of the Association for Women
              in Mathematics},
    VOLUME = {5},
      YEAR = {2026},
    NUMBER = {1},
     PAGES = {Paper No. 16, 51},
      ISSN = {2730-9657},
   MRCLASS = {35R35 (35B65 35N25 49N60 49Q10)},
  MRNUMBER = {5036204},
       DOI = {10.1007/s44007-026-00195-z},
       URL = {https://doi.org/10.1007/s44007-026-00195-z},
}

@article {EW_short_proof,
    AUTHOR = {Eberle, S. and Weiss, G. S.},
     TITLE = {Characterizing compact coincidence sets in the obstacle
              problem---a short proof},
   JOURNAL = {Algebra i Analiz},
  FJOURNAL = {Rossi\u iskaya Akademiya Nauk. Algebra i Analiz},
    VOLUME = {32},
      YEAR = {2020},
    NUMBER = {4},
     PAGES = {137--145},
      ISSN = {0234-0852},
   MRCLASS = {35J15},
  MRNUMBER = {4167864},
MRREVIEWER = {Jia-Feng\ Liao},
       DOI = {10.1090/spmj/1665},
       URL = {https://doi.org/10.1090/spmj/1665},
}

@article {DF_dge3,
    AUTHOR = {Di Fratta, Giovanni},
     TITLE = {The {N}ewtonian potential and the demagnetizing factors of the
              general ellipsoid},
   JOURNAL = {Proc. A},
  FJOURNAL = {Proceedings A},
    VOLUME = {472},
      YEAR = {2016},
    NUMBER = {2190},
     PAGES = {20160197, 7},
      ISSN = {1364-5021,1471-2946},
   MRCLASS = {78A25},
  MRNUMBER = {3522517},
       DOI = {10.1098/rspa.2016.0197},
       URL = {https://doi.org/10.1098/rspa.2016.0197},
}
\end{document}